\documentclass{article}
\pdfoutput=1
\usepackage{graphicx,url}
\usepackage[all]{xy}

\usepackage{amsfonts}
\usepackage{amssymb,amsmath,amscd}

\newtheorem{theorem}{Theorem}
\newtheorem{lemma}[theorem]{Lemma}

\newtheorem{corollary}[theorem]{Corollary}
\newtheorem{definition}{Definition}

\usepackage{amsfonts,amssymb,amsmath,,amscd}
\newcommand{\C}{\ensuremath{\mathbb{C}}}

\newcommand{\R}{\ensuremath{\mathbb{R}}}

\date{\today}
\begin{document}

\title{Angular momentum and Horn's problem}

\author{Alain Chenciner \& Hugo Jim\'enez-P\'erez
 \\
  \small Observatoire de Paris, IMCCE (UMR 8028), ASD\\
  \small 77, avenue Denfert-Rochereau, 75014 Paris, France\\
  \small \texttt{chenciner@imcce.fr, jimenez@imcce.fr}}

\maketitle

\abstract
We prove a conjecture made in \cite{C1}: given an $n$-body central configuration $X_0$ in the euclidean space $E$ of dimension $2p$,  let $Im{\cal F}$ be the set of ordered real $p$-tuples $\{\nu_1,\nu_2,\cdots,\nu_p\}$ such that 
$\{\pm i\nu_1,\pm i\nu_2,\cdots,\pm i\nu_p\}$
is the spectrum of the  angular momentum of some (periodic) relative equilibrium motion of $X_0$  in $E$. Then $Im {\cal F}$ is a convex polytope. The proof consists in showing that there exist two, generically $(p-1)$-dimensional, convex polytopes ${\cal P}_1$ and ${\cal P}_2$ in $\R^{p}$ such that ${\cal P}_1\subset Im{\cal F}\subset {\cal P}_2$ and that these two polytopes coincide. 

${\cal P}_1$, introduced in \cite{C1}, is the set of spectra corresponding to the hermitian structures $J$ on $E$ which are ``adapted" to the symmetries of the inertia matrix $S_0$; it is associated with Horn's problem for the sum of $p\times p$ real  symmetric matrices with spectra $\sigma_-$ and $\sigma_+$  whose union  is the spectrum  of $S_0$. 

${\cal P}_2$ is the orthogonal projection onto the set of  ``hermitian spectra" of the polytope 
${\cal P}$ associated with Horn's problem for the sum of $2p\times 2p$ real symmetric matrices having each the same spectrum as $S_0$. 

The equality ${\cal P}_1={\cal P}_2$ follows directly from a deep combinatorial lemma, proved in \cite{FFLP}, which characterizes those of the sums $C=A+B$ of two $2p\times 2p$ real symmetric matrices $A$ and $B$ with the same spectrum, which are hermitian for some hermitian structure.
\medskip

\section{Origin of the problem: $N$-body relative equilibria and their angular momenta}
We recall here the results of \cite{AC,C1,C2} which are needed in order to understand the mechanical origin of the purely algebraic conjecture solved in the present paper: 
given a configuration $x_0=(\vec r_1,\cdots,\vec r_N)\in E^N$ of $N$ punctual positive masses in the euclidean space $E$, a {\it rigid motion} of the configuration under Newton's attraction is a motion in which the mutual distances  $||\vec r_i-\vec r_j||$ between the bodies stay constant. It is proved in \cite{AC} (see also \cite{C2}) that such a motion is necessarily a {\it relative equilibrium}. This implies that the motion takes place in a space of even dimension $2p$, which can be supposed to coincide with $E$, and that, in a galilean frame fixing the center of mass at the origin, it is of the form 
$x(t)=(e^{\Omega t}\vec r_1,e^{\Omega t}\vec r_2,\cdots,e^{\Omega t}\vec r_N)$, where $\Omega$ is a $2p\times 2p$-antisymmetric endomorphism of the euclidean space $E$ which is non degenerate. 
Choosing an orthonormal basis where $\Omega$ is normalized, this amounts to saying that
there exists a hermitian structure on the space $E$ of motion and an orthogonal decomposition $E\equiv\C^p=\C^{k_1}\times\cdots\times\C^{k_r}$ such that
$$x(t)=(x_1(t),\cdots,x_r(t))=(e^{i\omega_1t}x_1,\cdots,e^{i\omega_rt}x_r),$$
where $x_m$ is the orthogonal projection on $\C^{k_m}$ of the $N$-body configuration $x$ and  the action of $e^{i\omega_mt}$ on $x_m$ is the diagonal action on each body of the projected configuration. Such quasi-periodic motions exist only for very special configurations, called {\it balanced configurations} (see \cite{AC,C2} for their characterization).
The most degenerate balanced configurations are the {\it central configurations} for which all the frequencies $\omega_i$ are the same; this means that $\Omega=\omega J$, with $J$ a hermitian structure on $E$, and the motion is
$$x(t)=(\vec r_1(t),\cdots,\vec r_N(t))=e^{i\omega t}x_0=(e^{i\omega t}\vec r_1,\cdots,e^{i\omega t}\vec r_N)$$ in the hermitian space $E\equiv\C^{p}$; in particular, it is periodic. In a space of dimension at most 3, $E$ is necessarily of dimension 2 and the configuration of any relative equilibrium is central. 
\smallskip

Given a configuration $x=(\vec r_1,\cdots,\vec r_N)$ and a configuration of velocities $y=\dot x=(\vec v_1,\cdots, \vec v_N)$, both with center of mass at the origin: 
$\sum_{k=1}^Nm_k\vec r_k=\sum_{k=1}^Nm_k\vec v_k=0$, the {\it angular momentum} of $(x,y)$ is the bivector ${\mathcal C}=\sum_{k=1}^Nm_k\vec r_k\wedge\vec v_k$. If we represent $x$ and $y$ by the $2p\times N$ matrices $X$ and $Y$ whose $i$th column are respectively made of the components $(r_{1i},\cdots,r_{2pi})$ and $(v_{1i},\cdots,v_{2pi})$ of $\vec r_i$ and $\vec v_i$ in an orthonormal basis of $E$ and if $M=\hbox{diag}(m_1,\cdots,m_N)$, this bivector is represented by the antisymmetric matrix
{\it (we use the french convention $^{t\!}X$ for the transposed of $X$)}
$$C=-XM^{t\!}Y+YM^{t\!\!}X\;\;\hbox{with coefficients}\;\; c_{ij}
=\sum_{k=1}^Nm_k(-r_{ik}v_{jk}+r_{jk}v_{ik}).$$
The dynamics of a solid body is determined by its {\it inertia tensor} (with respect to its center of mass), represented in the case of a point masses configuration $X$ by the symmetric matrix 
$$S=XM^{t\!\!}X\;\;\hbox{with coefficients}\;\; 
s_{ij}=\sum_{k=1}^Nm_kr_{ik}r_{jk},$$
whose trace is the {\it moment of inertia of the configuration $x$ with respect to its center of mass}.
In particular, the angular momentum of a relative equilibrium solution $X(t)=e^{t\Omega}X_0$
is represented by the antisymmetric matrix 
$C=S_0\Omega+\Omega S_0$, where $S_0=X_0M^{t\!\!}X_0$. Restricting to the case of central configurations, that is $\Omega=\omega J$, and making $\omega=1$, we consider in what follows the spectrum of $J$-skew-hermitian matrices of the form $S_0J+JS_0$ or, what amounts to the same, the spectrum of $J$-hermitian matrices\footnote{Notice that this is the same as the spectrum of the $J_0$-hermitian matrix 
$\Sigma=J_0^{-1}SJ_0+S$, where $S=RS_0R^{-1}$, which was considered in \cite{C1}.} of the form
$J^{-1}S_0J+S_0$.
\smallskip

\noindent {\it In the following, we identify $E$ with $\R^{2p}$ by the choice of some orthonormal basis. $\R^{2p}$ is endowed with its canonical basis $e_i=(0,\cdots,1,\cdots,0)$ and its canonical euclidean scalar product $x\cdot y=\sum_{i=1}^{2p}x_iy_i$; this allows identifying linear endomorphisms of $E=\R^{2p}$ and $2p\times 2p$ matrices with real coefficients. When we say that $J$ is a hermitian structure, we mean that the standard euclidean structure is given and that $J$ is a complex structure which is orthogonal.}

\section{The frequency map}
We recall the definition, given in \cite{C1}, of the {\it frequency map} ${\cal F}$ from the set of hermitian structures on $\R^{2p}$ to the positive Weyl chamber $W_p^+\subset \R^p$: 
given some $2p\times 2p$ real symmetric matrix $S_0$, we consider the mapping $J\mapsto J^{-1}S_0J+S_0$ from the space of hermitian structures on $E$ to the set of $2p\times 2p$ real symmetric matrices. We are only interested in the spectra of these matrices, hence choosing an orientation for $J$ is harmless and we shall consider only those of the form $J=R^{-1}J_0R$, where $J_0$ is the ``standard" structure defined by $J_0=\begin{pmatrix}0&-Id\\ Id&0\end{pmatrix}$ and $R\in SO(2p)$. Two elements $R'$ and $R''$ of $SO(2p)$ defining the same $J$ if and only if there exists an element $U\in U(p)$ such that $R''=UR'$, it follows that the space of (oriented) hermitian structures is identified to the homogeneous space $U(p)\backslash SO(2p)$. 
The symmetric matrix $J^{-1}S_0J+S_0$ is $J$-hermitian, that is, it commutes with $J$. This implies that its spectrum is real, of the form $\{\nu_1,\nu_2,\cdots,\nu_p\}$ if considered as a $p\times p$ complex matrix (for the identification of $\R^{2p}$ to $\C^p$ defined by $J$) and of the form $\{(\nu_1, \nu_2, \cdots, \nu_p), (\nu_1, \nu_2, \cdots, \nu_p)\}$ if considered as a $2p\times 2p$ real matrix (see the next section for the trivial proof).
\begin{definition}
The {\it frequency mapping} 
$${\mathcal F}:U(p)\backslash SO(2p)\to W_p^+=\{(\nu_1,\cdots\nu_p)\in\R^p,   \nu_1\ge\cdots\ge\nu_p\}$$ associates to each hermitian structure $J$ the ordered spectrum
$(\nu_1,\cdots,\nu_p)$ of the $J$-hermitian matrix
$J^{-1}S_0J+S_0$.
\end{definition}

\section{Hermitian spectra}
\begin{lemma} 
Let $C:\R^{2p}\to\R^{2p}$ be a symmetric endomorphism. The following assertions are equivalent:

1) There exists a hermitian structure $J=R^{-1}J_0R$ such that $C$ be $J$-hermitian (i.e. $JC=CJ$);

2) The spectrum $\sigma(C)$ of $C$ is of the form
$$\sigma(C)=\{(\nu_1, \nu_2, \cdots, \nu_p), (\nu_1, \nu_2, \cdots, \nu_p)\}.$$
\end{lemma}
\begin{proof} Let $J=R^{-1}J_0R$; the mapping $C$ is $J$-hermitian if and only if $RCR^{-1}$ is $J_0$-hermitian. This is equivalent to the existence of  an isomorphism $U\in U(p)\subset SO(2p)$ such that 
$$U^{-1}RCR^{-1}U=\hbox{diag}\bigl((\nu_1,\cdots,\nu_p),(\nu_1,\cdots,\nu_p)\bigr).$$
Conversely, the identity
$$R^{-1}CR=\hbox{diag}\bigl((\nu_1,\cdots,\nu_p),(\nu_1,\cdots,\nu_p)\bigr)$$
implies the commutation of $R^{-1}CR$ with $J_0$ and hence the commutation of $C$ with $J=R^{-1}J_0R$. 
\end{proof}
\smallskip

\noindent  {\bf Notations.} We shall call {\it hermitian} the spectra of this form and {\it the diagonal}  the linear subspace $\Delta$ of $W_{2p}^+$ defined by 
$$\Delta=\{(\mu_1\ge\mu_2\ge\cdots\ge\mu_{2p}),\; \mu_1=\mu_2,\, \mu_3=\mu_4,\, \, \cdots,\mu_{2p-1}=\mu_{2p}\}.$$
Hence the ordered hermitian spectra are the ones belonging to $\Delta$.

\section{Two convex polytopes}

Let $S_0:\R^{2p}\to\R^{2p}$ be a symmetric endomorphism with spectrum
$$\sigma(S_0)=\{\sigma_1\ge\sigma_2\ge\cdots\ge\sigma_{2p}\}.$$
To $S_0$ we associate the subsets ${\cal P}_1$ and ${\cal P}_2$ of $\R^p$ (in fact of the positive Weyl chamber $W_p^+$ of ordered real  $p$-tuples), defined as follows:
\smallskip

1) ${\cal P}_1$ is the set of ordered spectra 
$$\sigma(c)=\{\nu_1\ge\nu_2\ge\cdots\ge\nu_p\}$$ of symmetric endomorphisms $c$ of $\R^p$ of the form $c=a+b$, where $a:\R^p\to\R^p$ and $b:\R^p\to\R^p$ are arbitrary symmetric endomorphisms with respective spectra
$$\sigma(a)=\sigma_-:= \{\sigma_1,\sigma_3,\cdots,\sigma_{2p-1}\},\quad \sigma(b)=\sigma_+:=\{\sigma_2,\sigma_4,\cdots,\sigma_{2p}\} ;$$ 

2) ${\cal P}_2$ is the set of $p$-tuples $\{\nu_1\ge\nu_2\ge\cdots\ge\nu_p\}$ such that
$$\{(\nu_1,\nu_2\cdots,\nu_{p}),(\nu_1,\nu_2\cdots,\nu_{p})\}$$
is the spectrum of some symmetric endomorphism $C$ of $\R^{2p}$ of the form $C=A+B$, where $A:\R^{2p}\to\R^{2p}$ and $B:\R^{2p}\to\R^{2p}$ are arbitrary symmetric endomorphisms with the same spectrum 
$$\sigma(A)=\sigma(B)=\sigma(S_0).$$ 
In other words, identifying canonically the diagonal $\Delta$ with $\R^p$, one can write 
$${\cal P}_2={\cal P}\cap\Delta,$$ 
where ${\cal P}$ is the $(2p-1)$-dimensional Horn polytope which describes the ordered spectra of sums $C=A+B$ of $2p\times 2p$ real symmetric  matrices $A,B$ with the same spectrum as $S_0$.  
\begin{lemma}
${\cal P}_1$ and  ${\cal P}_2$ are both contained in the hyperplane of $\R^p$ with equation 
$$\sum_{i=1}^p\nu_i=\sum_{j=1}^{2p}\sigma_j.$$  
Moreover, ${\cal P}_1$ and ${\cal P}_2$ are both $(p-1)$-dimensional convex polytopes and
$${\cal P}_1\subset Im {\cal F} \subset{\cal P}_2.$$
\end{lemma}

\begin{proof} The first identity comes from the additivity of the trace function. The fact that both ${\cal P}_1$, ${\cal P}$ and  hence ${\cal P}_2={\cal P}\cap\Delta$, are convex polytopes is a general fact coming from the interpretation of the Horn problem as a moment map problem. Finally, the second inclusion comes from the very definition of ${\cal F}$ and the first comes from Lemma 1 and the following identity, where $\sigma_-$ and $\sigma_+$ are considered as $p\times p$ diagonal matrices and $\rho\in SO(p)$:
$$
\begin{cases}
&\begin{pmatrix}
\sigma_-&0\\
0&\sigma_+
\end{pmatrix}+
\begin{pmatrix}
0&-\rho^{-1}\\
\rho&0
\end{pmatrix}^{-1}
\begin{pmatrix}
\sigma_-&0\\
0&\sigma_+
\end{pmatrix}
\begin{pmatrix}
0&-\rho^{-1}\\
\rho&0
\end{pmatrix}\\
=&
\begin{pmatrix}
\sigma_-+\rho^{-1}\sigma_+\rho&0\\
0&\rho\sigma_-\rho^{-1}+\sigma_+
\end{pmatrix}\cdot
\end{cases}
$$
\end{proof}

\noindent{\bf Remark.} The choice of the partition $\sigma=\sigma_-\cup\sigma_+$ of $\sigma$ is dictated by the following theorem, which proves that any other partition of $\sigma$ into two subsets with $p$ elements will lead to a smaller polytope ${\cal P}_1$:
\begin{theorem}[\cite{FFLP} Proposition 2.2]  Let $A$ and $B$ be $p\times p$ Hermitian matrices. Let  $\sigma_1\ge\sigma_2\ge\cdots\ge\sigma_{2p}$ be the eigenvalues of $A$ and $B$ arranged in descending order. Then there exist Hermitian matrices $\tilde A$ and $\tilde B$ with eigenvalues $\sigma_1\ge\sigma_3\ge\cdots\ge\sigma_{2p-1}$ and $\sigma_2\ge\sigma_4\ge\cdots\ge\sigma_{2p}$ respectively, such that $\tilde A+\tilde B=A+B$.
\end{theorem}
This was used in \cite{C1} to prove that ${\cal P}_1= Im {\cal A}$ is the image under the frequency map ${\cal F}$ of the {\it adapted} hermitian structures, i.e. those $J$ which send some $p$-dimensional subspace of $\R^{2p}$ generated by eigenvectors of $S_0$ onto its orthogonal.

\section{The projection property}

In this section, we prove the 
\begin{theorem}The two polytopes ${\cal P}_1$ and ${\cal P}_2$ coincide.
\end{theorem}
\begin{corollary} ${Im {\cal F}}={\cal P}_1= Im {\cal A}$. In other words, the image by the frequency map ${\cal F}$ of the adapted structures is already the full image $Im {\cal F}$. 
\end{corollary}

\noindent We need recall the inductive definition of the Horn inequalities which define the Horn polytopes (see \cite{F}).  For a sum $a+b=c$ of symmetric $p\times p$ matrices  
with respective (ordered in decreasing order) spectra 
$$\alpha=(\alpha_1,\cdots,\alpha_p),\;  \beta=(\beta_1,\cdots,\beta_p),\; \gamma=(\gamma_1,\cdots,\gamma_p),$$
they read
$$
(^*IJK)\quad\quad\quad \forall r<p,\; \forall (I,J,K)\in T^p_r,\quad \sum_{k\in K}\gamma_k\le \sum_{i\in I}\alpha_i+\sum_{j\in J}\beta_j,
$$
where $T^p_r$ (notation of \cite{F}, noted $LR^p_r$ by reference to Littlewood-Richardson coefficients in \cite{FFLP}) is defined as follows:
let $U^p_r$ be the set of triples $(I,J,K)$ of subsets of cardinal $r$ of $\{1,2,\cdots, p\}$ 
such that
$$\sum_{i\in I}i+\sum_{j\in J}j=\sum_{k\in K}k+\frac{r(r+1)}{2}.$$
Then set $T^p_1=U^p_1$ and define recursively $T^p_r$ by
$$
T^p_r=
\begin{bmatrix}
(I,J,K)\in U^p_r, \forall s<r,\, \forall (F,G,H)\in T^r_s, \\
\sum_{f\in F}i_f+\sum_{g\in G}j_g\le \sum_{h\in H}k_h+\frac{s(s+1)}{2}
\end{bmatrix}
$$

\noindent An immediate computation gives the following
\begin{lemma}
Let
$$
\begin{cases}
I_2&=(2i_1-1,2i_2-1,\cdots,2i_r-1,2j_1,2j_2,\cdots,2j_r),\\
J_2&=(2i_1-1,2i_2-1,\cdots,2i_r-1,2j_1,2j_2,\cdots,2j_r),\\
K_2&=(2k_1-1,2k_1,2k_2-1,2k_2,\cdots,2k_r-1,2k_r),
\end{cases}
$$

Then $(I_2,J_2,K_2)\in U^{2p}_{2r}$.
\end{lemma}
\begin{proof} It suffices to check that 
$$2\left[\sum_{i\in I}(2i-1)+\sum_{j\in J}(2j)\right]=\sum_{k\in K}(2k-1)+\sum_{k\in K}2k+\frac{2r(2r+1)}{2}\cdot$$
\end{proof}

Now, comes the key fact:
\begin{theorem}[\cite{FFLP}, lemma 1.18] For any triple $(I,J,K)$ in $T^p_r$, the triple $(I_2,J_2,K_2)$ is in $T^{2p}_{2r}$
\end{theorem}
The proof of this theorem, which concerns the so-called ``domino-decompo\-sable Young diagrams",  is based on a version of the Littlewood-Richardson rule due to Carr\'e and Leclerc  \cite{CL}.
\smallskip

\noindent It implies that, for any 
a sum $A+B=C$ of real symmetric $2p\times 2p$ matrices  
with respective (ordered in decreasing order) spectra 
$$\hat\alpha=(\hat\alpha_1,\cdots,\hat\alpha_{2p}),\;  \hat\beta=(\hat\beta_1,\cdots,\hat\beta_{2p}),\; \hat\gamma=(\hat\gamma_1,\cdots,\hat\gamma_{2p}),$$ 
$(^*I_2,J_2,K_2)$ holds, that is
$$
\sum_{k\in K}(\hat\gamma_{2k-1}+\hat\gamma_{2k})\le \sum_{i\in I}(\hat\alpha_{2i-1}+\hat\beta_{2i-1})+\sum_{j\in J}(\hat\alpha_{2j}+\hat\beta_{2j}).
$$
In particular, if 
$$\hat\alpha=\hat\beta=\sigma=(\sigma_1,\sigma_2,\cdots,\sigma_{2p}),$$
we get that
$$\sum_{k\in K}\frac{\hat\gamma_{2k-1}+\hat\gamma_{2k}}{2}\le \sum_{i\in I}\sigma_{2i-1}+\sum_{j\in J}\sigma_{2j}.$$
Note that the mapping 
$$(\hat\gamma_1,\hat\gamma_2,\cdots,\hat\gamma_{2p-1},\hat\gamma_{2p})\mapsto (\frac{\hat\gamma_1+\hat\gamma_2}{2},\frac{\hat\gamma_1+\hat\gamma_2}{2},\cdots, \frac{\hat\gamma_{2p-1}+\hat\gamma_{2p}}{2},\frac{\hat\gamma_{2p-1}+\hat\gamma_{2p}}{2})$$
 is the orthogonal projection of the ordered set 
$(\hat\gamma_1,\hat\gamma_2,\cdots,\hat\gamma_{2p-1},\hat\gamma_{2p})$ on the {\it diagonal} 
$\Delta$ of $\R^{2p}$ defined by the equations
$\hat\gamma_1=\hat\gamma_2,\cdots,\hat\gamma_{2p-1}=\hat\gamma_{2p}$, that is on the subset of ``hermitian" spectra. Hence a paraphrase of the above theorem is

\begin{theorem} Let $C=A+B$ be the sum of two $2p\times 2p$ real symmetric matrices with the same spectrum $\{\sigma_1\ge\sigma_2\ge\cdots\ge\sigma_{2p}\}$. 

\noindent If $\{\nu_1\ge\cdots\ge\nu_p\}$ is the orthogonal projection on the diagonal $\Delta\equiv\R^p$ of the spectrum 
$\{\hat\gamma_1\ge\hat\gamma_2\ge\cdots\ge\hat\gamma_{2p}\}$ of $C$, that is if $\nu_k=\frac{\hat\gamma_{2k-1}+\hat\gamma_{2k}}{2}$, the triple of ordered spectra
$$\alpha=(\sigma_1,\sigma_3,\cdots,\sigma_{2p-1}),\; \beta=(\sigma_2,\sigma_4,\cdots,\sigma_{2p}),\; \gamma=(\nu_1,\nu_2,\cdots,\nu_p)$$
satisfies the Horn inequality $(^*I,J,K)$. 
\end{theorem}

\noindent This implies the following extremal property of the subset of ``hermitian" spectra:
\begin{corollary}
The orthogonal projection on the diagonal $\Delta$ of the $(2p-1)$-dimensional Horn polytope  ${\cal P}\subset \R^{2p}$ associated with the spectra 
$$\sigma(A)=\sigma(B)=\{\sigma_1\ge\sigma_2\ge\cdots\ge\sigma_{2p}\}$$
coincides with the $(p-1)$-dimensional Horn polytope ${\cal P}_1\in\R^p$ associated with the spectra
$$\sigma(a)=(\sigma_1,\sigma_3,\cdots,\sigma_{2p-1}),\; \sigma(b)=(\sigma_2,\sigma_4,\cdots,\sigma_{2p}).$$
\end{corollary}

\noindent In particular, the intersection ${\cal P}_2={\cal P}\cap\Delta$ corresponding to the hermitian spectra, that is those such that
$\hat\gamma_1=\hat\gamma_2,\cdots,\hat\gamma_{2p-1}=\hat\gamma_{2p}$, coincides with ${\cal P}_1$. Indeed, ${\cal P}_2$ coincides with the projection of ${\cal P}$, which itself coincides with ${\cal P}_1$. 
\smallskip

\noindent {\bf Remark.}  The equality $Im {\cal F}={\cal P}_2$ implies the following
\begin{corollary}Let $C:\R^{2p}\to\R^{2p}$ be the sum $C=A+B$ of two symmetric endomorphisms $A$ and $B$ with the same spectrum $\sigma(A)=\sigma(B)=\sigma(S_0)$. Then $C$ is $J$-hermitian for some hermitian structure $J$ on $\R^{2p}$ if and only if it is conjuguate by an element of $SO(2p)$ to a matrix of the form $S_0+\tilde J^{-1}S_0\tilde J$,
where $\tilde J$ is a hermitian structure on $R^{2p}$.  
\end{corollary}

\noindent {\bf Note.} The proof of the above results has been written by the first author after he was convinced by the numerical experiments made by the second author that the equality ${\cal P}_1={\cal P}_2$ was plausible when $p=3$ and more precisely that not only the intersection 
${\cal P}_2={\cal P}\cap\Delta$ but also the orthogonal projection of the Horn polytope ${\cal P}$ on 
$\Delta$ was contained in ${\cal P}_1$ after the canonical identification of $\Delta$ with $\R^p$. This led first to a proof when $p=2$ or $3$, obtained by coping directly with Horn's inequalities and then to the discovery that the general case followed from a lemma which turned out to be exactly the lemma 1.18 of \cite{FFLP}. The numerical experiments are described in the next section.

\section{Numerical experiments} 
The numerical checking of the conjecture that $Im {\cal F}={\cal P}_1$, was made on the matrix 
$S_0= \frac{1}{32}{\rm diag}\left\{ 13,8,5,3,2,1 \right\}$,
whose spectrum satisfies the inequalities in \cite{C1} (section 8).
We wrote a program in TRIP \cite{GL11} producing different
rotation matrices $R\in SO(2p)$ in a random way.
Starting from the canonical basis $\xi=\{\xi_1,\dots,\xi_m\},$ $m=p(2p-1)$, 
of $\mathfrak{so}(2p)$, we created a list containing the $m$ one-parameter
subgroups $G_i(t)= e^{t\xi_i}\subset SO(2p)$. We created a 
second list of $m$ random values $[t_i]_{i=1}^m$ 
and a random permutation $[1,2,\dots,m]\to [i_1,i_2,\dots ,i_m]$. 
The random rotation matrix 
was defined as
\begin{eqnarray}
    R = \prod_{j=1}^m G_{i_j}(t_j),\quad 0\leq t_i \leq 2\pi.
    \label{eqn:R}
\end{eqnarray}
The program which plots $\mathcal P_1$ is very simple (the fact that we replaced the conjugation of $J_0$ by the conjugation of $S_0$ is immaterial and comes from the formulation of the conjecture in \cite{C1}): \\
\begin{center}
  \begin{tabular}[ht]{|l|}
    \hline
      \verb+create+ $S_0$ \verb+and+ $J_0$\\
      \verb+for+ $i=1$ \verb+to+ $N_{max}$ \verb+do+\\
      $\quad$ \verb+create+ $R$ \verb+and+ $R^{-1}$;\\ 
      $\quad$ $S = RS_0R^{-1}$;\\
      $\quad$ $C = S - J_0 S J_0$;\\
      $\quad$ \verb+lst = eigenvalues(+$C$\verb+)+;\\
      $\quad$ \verb+plot ( lst[5], lst[3], lst[1] )+;\\
      \verb+end for+.\\
    \hline
  \end{tabular}
\end{center}
We have assigned the value $N_{max}=25000$ obtaining the results shown in Figure 
\ref{fig:AA}. The figure shows also the simplex $\gamma_1+\gamma_2+\gamma_3=1$
and its intersection with $W_3^+$.
\begin{figure}[h]
    \centering
    \includegraphics[scale=0.375]{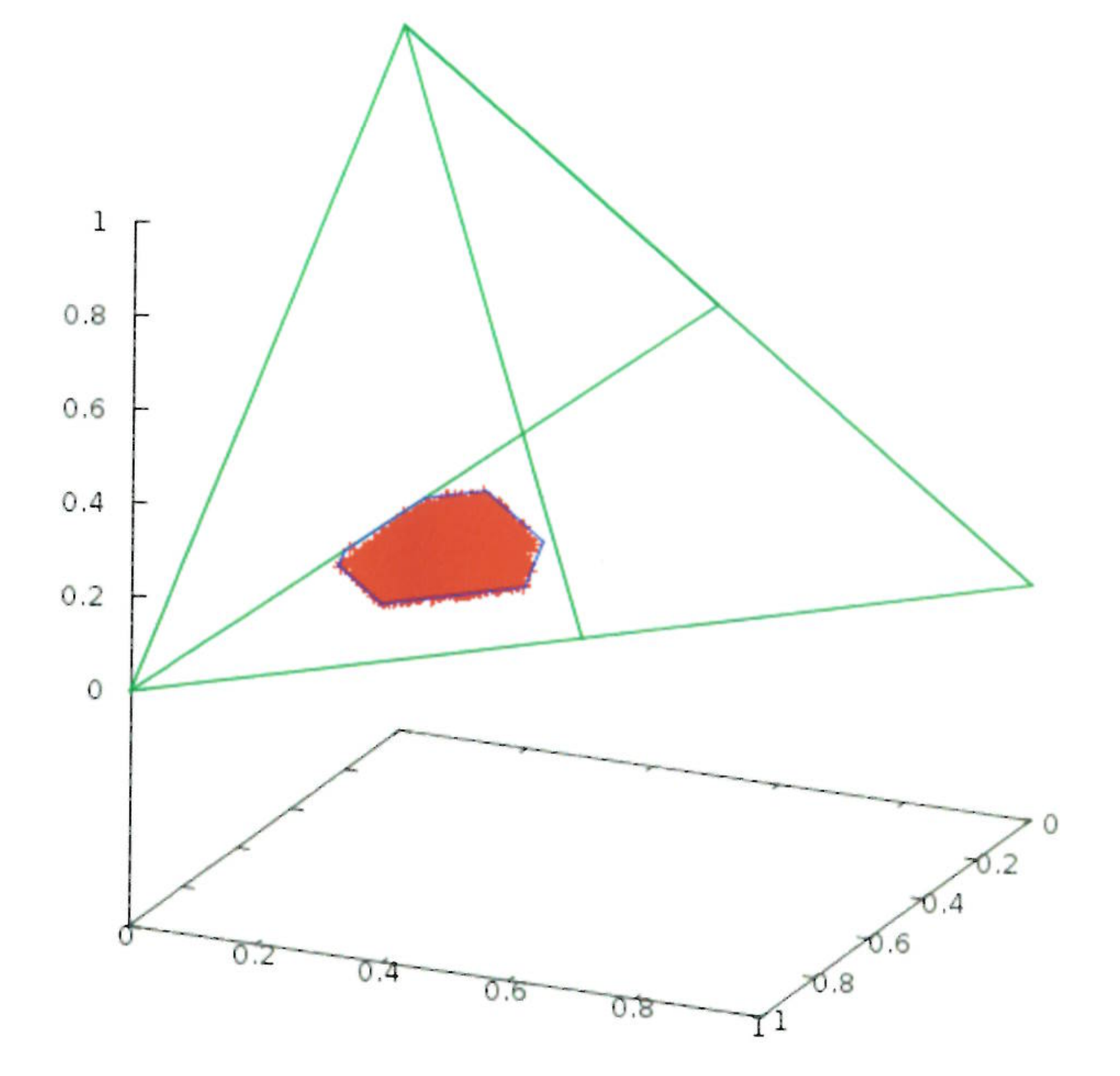}
    \caption{${\cal P}_1=Im {\cal A}$: 25000 random adapted hermitian structures}
    \label{fig:AA}
\end{figure}

\noindent The modified algorithm to estimate the shape of $\mathcal P_2=\mathcal 
P\cap\Delta$ in $W^+_3$ is similar. For a random $R$ in $SO(6)$, the ordered spectrum
$spec(C)=\left( \gamma_1,\dots,\gamma_6 \right)$ of $C=S_0+R^{-1}S_0R$
is  projected orthogonally onto the diagonal $\Delta$ by the map $\pi_\Delta:\mathbb R^6 \to \Delta$: \begin{eqnarray*}
 (\gamma_1,\gamma_2,\dots,\gamma_6)&\mapsto& \left( \frac{\gamma_1+\gamma_2}{2},
 \frac{\gamma_3+\gamma_4}{2},
 \frac{\gamma_5+\gamma_6}{2}
    \right).
\end{eqnarray*}
At first, when $spec(C)$ was $\varepsilon$-close to $\Delta$ \emph{i.e.}, if 
$\sum_{k=1}^3 |\gamma_{2k-1}-\gamma_{2k}|^2 < 2\varepsilon^2$ for 
$\varepsilon$ small, the projected point was plotted in green; otherwise it was 
plotted in red. No particular pattern was found for the green points meanwhile the red ones were all contained in $\mathcal P_1$. 
The algorithm to plot $\pi_\Delta(\mathcal P)$ is

\begin{center}
  \begin{tabular}[ht]{|l|}
    \hline
      \verb+create+ $S_0$\\
      \verb+for+ $i=1$ \verb+to+ $N_{max}$ \verb+do+\\
      $\quad$ \verb+create+ $R$ \verb+and+ $R^{-1}$;\\ 
      $\quad$ $C = S_0 + R^{-1} S_0 R$;\\
      $\quad$ \verb+lst = eigenvalues(+$C$\verb+)+;\\
      $\quad$ \verb+sort( lst )+;\\
      $\quad$ \verb+plot+ $\left( 
      \frac{lst[6]+lst[5]}{2}, 
      \frac{lst[4]+lst[3]}{2}, 
      \frac{lst[2]+lst[1]}{2}\right)$\verb+;+\\ 
      \verb+end for+.\\
    \hline
  \end{tabular}
\end{center}

\noindent The results of the projection $\pi_\Delta(spec(C))$
for $50000$ random rotation matrices are shown in Figure \ref{fig:BB}. 

\begin{figure}[ht]
    \centering
    \includegraphics[scale=0.4]{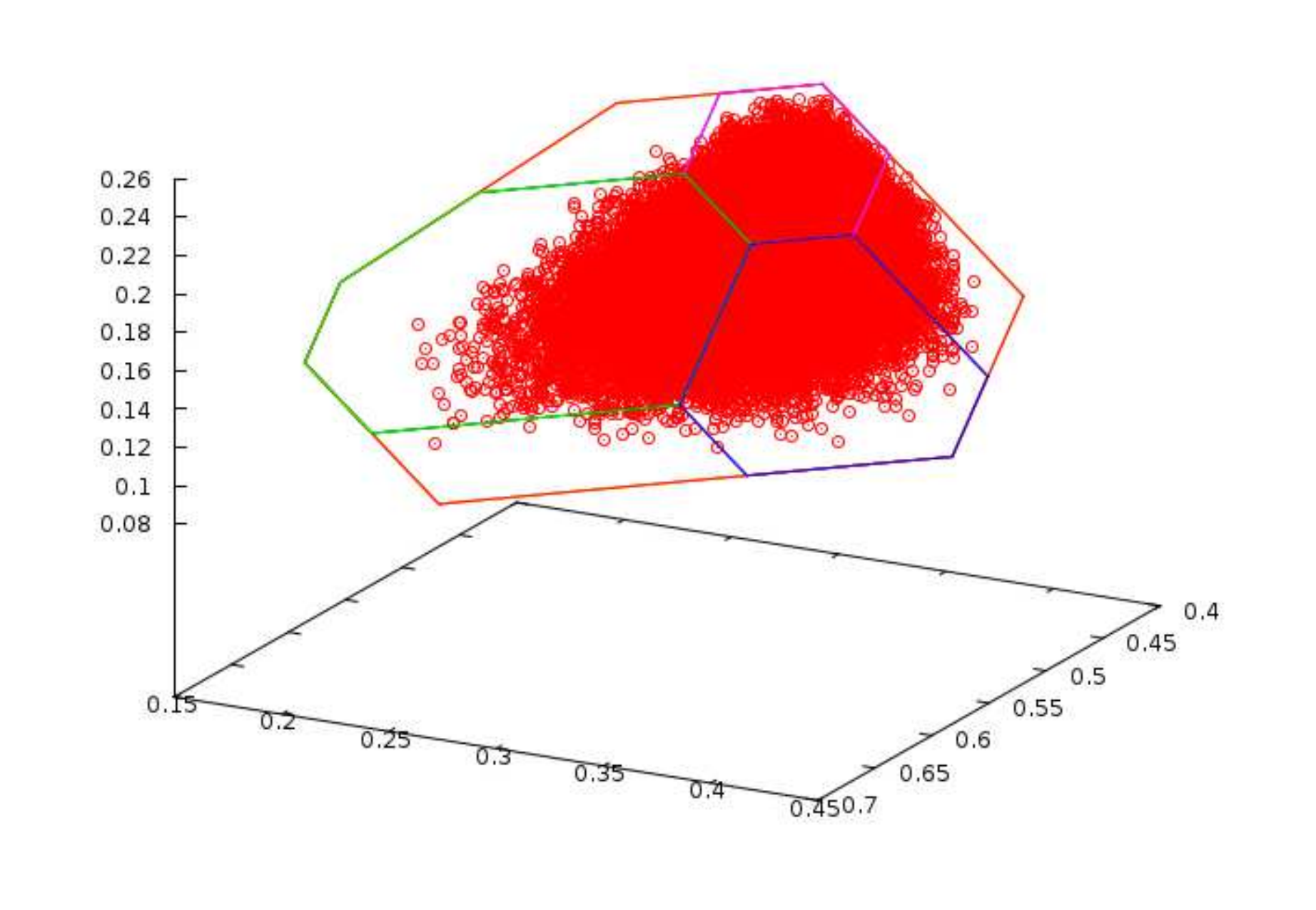}
    \caption{Projection of ${\cal P}$: 50000 random rotations}
    \label{fig:BB}
\end{figure}

The matrix $S_0$ and hence the polytope $\mathcal P_1$ are the same as in Figure~\ref{fig:AA} (the interior lines correspond to the polytopes associated to different partitions of the spectrum of $S_0$, as depicted in the corresponding figure in \cite{C1}). Recall that the polytope ${\cal P}$ has dimension 5; this explains that in order to get a better filling one should have taken many more points. This was not done because the evidence was sufficiently convincing to ask for a proof.

\section{Acknowledgements}  Warm thanks to Sun Shanzhong and Zhao Lei for their interest in this work and numerous discussions.

\end{document}